\documentclass[12pt,a4paper]{article}

\usepackage{amsthm}
\usepackage{amssymb}
\usepackage{amsmath}

\usepackage{comment}
\usepackage[left=2cm,right=2cm,top=2cm,bottom=2cm]{geometry}
\usepackage{enumitem,multicol}
\usepackage{float}
\usepackage{amsfonts}
\usepackage{cite}
\usepackage{graphicx}
\usepackage{makeidx}
\makeindex

\newcommand{\so}{\operatorname{so}\nolimits}

\def\N{\mathbb{N}}
\def\R{\mathbb{R}}

\def\g{\mathfrak g}

\def\a{\alpha}
\def\lam{\lambda}
\def\f{\varphi}

\newcommand{\spann}{\operatorname{span}\nolimits}
\def\ad{\operatorname{ad}\nolimits}
\def\id{\operatorname{id}}

\def\so{\operatorname{so}}

\def\rank{\operatorname{rank}}

\def\const{\operatorname{const}} 
\def\Ad{\operatorname{Ad}} 
\renewcommand\Vec{\operatorname{Vec}} 

\def\vh{\vec{h}}
\def\vg{\vec{g}}

\newcommand{\be}[1]{\begin{equation}\label{#1}}
\newcommand{\ee}{\end{equation}}
\newcommand{\eq}[1]{$(\protect\ref{#1})$}

\newtheorem{theorem}{Theorem}
\newtheorem*{remark}{Remark}

\newtheorem{corollary}{Corollary}

\numberwithin{equation}{section}

\graphicspath{{fig/}}

 \author{Yu. Sachkov
\thanks{
Program Systems Institute, Pereslavl-Zalessky, Russia, 
        {\tt\small yusachkov@gmail.com}}%
}

\title{Co-adjoint orbits and time-optimal problems
\\for two-step free-nilpotent Lie groups
}

\begin{document}

\maketitle

\begin{abstract}
We describe co-adjoint orbits and Casimir functions for two-step free-nilpotent Lie algebras.
The symplectic foliation consists of affine subspaces of the Lie coalgebra of different dimensions.

Further, we consider left-invariant time-optimal problems on two-step Carnot groups, for which the set of admissible velocities is a strictly convex compactum in the first layer of the Lie algebra containing the origin in its interior. We describe integrals for the vertical subsystem of the Hamiltonian system of Pontryagin maximum principle.
Further, we describe constancy and periodicity of solutions to this subsystem and controls, and characterize its flow, for two-dimensional co-adjoint orbits.
\end{abstract}

\tableofcontents

\newpage

\section{Introduction}\label{sec:intro}

We consider linear-in-controls time-optimal left-invariant problems on step 2 Carnot groups, with a strictly convex control set. In particular, this class of problems contains sub-Riemannian \cite{mont,  notes, ABB} and sub-Finsler \cite{ber, ber2, BBLDS, ALDS, ali-charlot} problems. Our aim is to characterize extremal controls.

It is enough to consider the free-nilpotent cases since any step-2 Carnot group is a quotient of a step-2
free-nilpotent Lie group (with the same number of generators) and, moreover, every minimizing curve lifts to a
minimizing curve.
Indeed,
see Theorem 4.2 in \cite{ledonne-speight} 
for the existence of free Carnot groups;
see Corollary 2.11 in \cite{ledonne-rigot} 
for the fact that the quotient is a submetry and therefore geodesics lift to geodesics.

  For a two-step free-nilpotent Lie group, we describe the symplectic foliation and Casimir functions. The symplectic foliation consists of affine subspaces in the Lie coalgebra of dimensions 0, 2, \dots $2[k/2]$, where $k$ is the number of generators of the Lie algebra. 
	Casimir functions are all linear, except one function that is a homogeneous polynomial of order $(k+1)/2$, for odd~$k$. 
	These objects are important since Casimir functions are integrals, and symplectic leaves are invariant sets of the Hamiltonian system of Pontryagin maximum principle for left-invariant optimal control problems on Lie groups.

Further, we consider a left-invariant time-optimal problem, for which the set of admissible velocities is a strictly convex compactum in the first layer of the Lie algebra containing the origin in its interior. We apply Pontryagin maximum principle, and for the vertical subsystem of the corresponding Hamiltonian system we describe integrals. 
We characterize constancy and periodicity of solutions  to this subsystem and  describe its flow for two-dimensional co-adjoint orbits. 

This paper generalizes similar results for the number of generators $k = 2, 3$ obtained respectively by V. Berestovskii \cite{ber2} and the author \cite{sf36}. 

\medskip

This paper has the following structure. In Sec.~\ref{sec:problem} we recall the definitions of free Carnot algebras and groups, and  state the time-optimal problem. In Sec.~\ref{sec:foliation} we describe explicitly the symplectic foliation (decomposition into co-adjoint orbits) of the Lie coalgebra. In Sec.~\ref{sec:Casimir} we compute Casimir functions. Then we start to study the time-optimal problem. In Sec.~\ref{sec:PMP} we apply Pontryagin maximum principle. And in Sec.~\ref{sec:integrals} we describe integrals of the vertical subsystem of the Hamiltonian system of PMP, and prove the constancy and periodicity properties of solutions to this system for two-dimensional co-adjoint orbits. In the final Section~\ref{sec:final} we comment on results of this paper and suggest some questions for future research.

\section{Optimal control problem}\label{sec:problem}
Let $\g$  be the step 2 free-nilpotent Lie algebra with $k \geq 2$ generators:
\begin{align}
&\g = \g^{(1)} + \g^{(2)}, \nonumber& \\
&\g^{(1)} = \spann \{ X_i ~|~ i = 1, \dots, k \}, \nonumber&  \\
&\g^{(2)} = \spann \{ X_{ij} ~|~ 1 \le i < j \le k \}, \nonumber & \\
&[X_i, X_j] = X_{ij}, \quad \ad X_{ij} = 0, \quad 1 \le i < j \le k, &\label{tab1} \\
&\dim \g = k(k + 1)/2.& \nonumber 
\end{align}
Let $G$ be the connected simply connected Lie group with the Lie algebra $\g$. We will think of~$X_i, X_{ij}$ as left-invariant vector fields on $G$.

A model of vector fields $X_i, X_{ij}$  on $G \cong \mathbb{R}^{k(k+1)/2} = \{ (x_1, \dots, x_k;~x_{12}, \dots, x_{(k - 1)k}) \}$ is given by
\begin{align*}
&X_i = \frac{\partial}{\partial x_i} - \sum_{j >i} \frac{x_j}{2}  \frac{\partial}{\partial x_{ij}} + 
\sum_{j <i} \frac{x_j}{2}  \frac{\partial}{\partial x_{ji}}, \quad i = 1, \dots, k,\\
&X_{ij} = \frac{\partial}{\partial x_{ij}}, \quad 1 \le i < j \le k,
\end{align*}
here we follow 
Section 2.2 in \cite{ledonne-speight}.

Let $U \subset \mathbb{R}^k$ be a compact convex set containing the origin in its interior.
We consider the following time-optimal problem \cite{notes}: 
\begin{align}
&\dot{g} = \sum^k_{i=1} u_i X_i, \quad g \in G, \quad u = (u_1, \dots, u_k) \in U, \label{p21} &\\
&g(0)  = \id, \quad g(t_1) = g_1, \label{p22} &\\
&t_1 \to \min. \label{p23}&
\end{align}
If $U = -U$, we obtain a sub-Finsler problem  \cite{ber, ber2, BBLDS, ALDS}, and if $U$ is an ellipsoid centered at the origin, we obtain a sub-Riemannian problem \cite{mont, notes, ABB}. 

In the case $k = 2$, $G$ is the Heisenberg group, and solution to problem \eqref{p21}--\eqref{p23} was obtained by H. Busemann~\cite{buseman} and V. Berestovskii~\cite{ber2}.  The case $k = 3$ was studied in~\cite{sf36}. In the both cases $k = 2, 3$, extremal controls are constant or periodic. The main goal of this work is to generalize this result to arbitrary $k$, for two-dimensional co-adjoint orbits. 

The sub-Riemannian case $U = \{\sum_{i=1}^k u_i^2 \leq 1\}$ was first considered by R.Brockett~\cite{brockett}, and was completely solved for $k=3$ by O.Myasnichenko~\cite{myasnich}. Some partial results for $k=4$ were obtained by L. Rizzi and U. Serres~\cite{rizzi-serres}.

Existence of optimal solutions in problem \eqref{p21}--\eqref{p23} follows in a standard way  from the Rashevsky-Chow and Filippov theorems \cite{notes}.

\section{Symplectic foliation}\label{sec:foliation}
Before our study of extremals for problem \eqref{p21}--\eqref{p23}, we consider Casimirs and symplectic foliation (decomposition into coadjoint orbits) on the dual $\g^*$  of the Lie algebra $\g$ \cite{kirillov}. This is important for our study of extremals for the problem.

Let $p^0 \in \g^*$. Denote by  $S_{p^0}$ the {\em  symplectic leaf} ({\em co-adjoint orbit}\,) through the point $p^0$:
$$
S_{p^0} = \left\{ \Ad^*_{g^{-1}}(p^0) \mid g \in G\right\} \subset \g^*.
$$
We obtain an explicit description of a leaf  $S_{p^0}$ below in Th. \ref{th:fol}.

Introduce linear on fibers of  the cotangent bundle $T^*G$ Hamiltonians corresponding to the basic left-invariant vector fields on $G$:
$$ h_i (\lambda) = \langle \lambda, X_i \rangle, \quad h_{ij}(\lambda) = \langle \lambda, X_{ij} \rangle, \quad \lambda \in T^*G. $$
Product rule \eqref{tab1} for Lie bracket  implies the following multiplication table for Poisson bracket:
\be{tab2}
\{ h_i, h_j \} = h_{ij}, \quad \{ h_{ij}, h_l \} = \{ h_{ij}, h_{lm} \} = 0. 
\ee
The Hamiltonians $h_i$, $h_{ij}$ will be considered as coordinates on the dual $\g^*$ of the Lie algebra $\g$.

Notice that the {\em Poisson bivector} (i.e., the matrix of pairwise Poisson brackets of the basis Hamiltonians $h_i$, $h_{ij}$) is determined by the skew-sym\-met\-ric matrix
\be{M}
M = (h_{ij}) = 
\begin{pmatrix}
0& h_{12}& \ldots &h_{1k}\\
-h_{12}& 0 & \ldots & h_{2k}\\
\vdots& \vdots& \vdots& \vdots\\
-h_{1k}& -h_{2k}& \ldots& 0
\end{pmatrix}
\in \so (k). 
\ee
It is well known \cite{kirillov} that  $S_{p^0}$ is a symplectic manifold with
\be{dimS}
\dim S_{p^0} = \rank M_{p^0}.
\ee

In order to describe the leaves $S_{p^0}$ explicitly,  consider a linear function, 
for a vector $a = (a_1, \dots, a_k) \in \mathbb{R}^k$:
$$ I_a  =  \sum^k_{i=1} a_i h_i, \qquad I_a \ : \ \g^* \to \R. $$
Further, for any 
  $p^0 \in \g^*$  consider an affine subspace 
\be{Lp0}
L_{p^0} = \{ p \in \g^* \mid h_{ij}(p) = h_{ij}(p^0), \ I_a(p) = I_a(p^0) \ \forall a \in \ker M_{p^0}\} \subset \g^*,
\ee
where $M_{p^0} = (h_{ij}(p^0))$  is the corresponding matrix~\eq{M}.
It is obvious that
$$\dim L_{p^0}  = \rank M_{p^0},$$
thus, in view of \eq{dimS}, we have $\dim L_{p^0}= \dim S_{p^0}$. This coincidence is not accidental.

\begin{theorem}\label{th:fol}
For any  $p^0 \in \g^*$ we have $S_{p^0} = L_{p^0}$. 
Thus co-adjoint orbits are affine subspaces of $\g^*$ of dimensions $0, 2, \dots, 2[k/2]$. 
\end{theorem}
\begin{proof}
The leaf  $S_{p^0} \subset \g^*$ is a connected smooth manifold of dimension $\rank M_{p^0}$. The affine subspace $L_{p^0} \subset \g^* $ has the same dimension.  We now  prove that $S_{p^0} \subset L_{p^0}$. 

Take any $p = \Ad^*_{g^{-1}}(p^0) \in S_{p^0}$. We have to prove that
\be{hijp}
h_{ij}(p) = h_{ij}(p^0), \qquad I_a(p) = I_a(p^0) \ \forall a \in \ker M_{p^0}.
\ee

\medskip
(1) Consider first  a special case: let
\be{geTX}
g = e^{TX}, \text{ where } X = X_i \text{ or } X_{ij}.
\ee
Then 
$$
p = \Ad^*_{e^{-TX}} (p^0) = \left(e^{-TX}\right)^*(p^0) = e^{T\vh}(p^0),
$$
where $h(\lam) = \langle \lam, X \rangle \in C^{\infty}(T^*G)$ is the linear on fibers of $T^*G$ Hamiltonian corresponding to~$X$, and  $\vh \in \Vec(\g^*)$ is the vertical part of the Hamiltonian vector field with the Hamiltonian~$h$~\cite{notes}.
Let us denote $p(t) = e^{t \vh}(p^0)$, $t \in [0, T]$, and prove~\eq{hijp}.

We have
$$
\frac{d}{dt} h_{ij}(p(t)) = \frac{d}{dt} h_{ij} \circ e^{t \vh} (p(t)) = \vh h_{ij}(p(t))  = \{h, h_{ij}\} (p(t)) = 0,
$$
thus $h_{ij}(p) = h_{ij}(p^0)$.

Further, 
\be{dI}
\frac{d}{dt} I_a (p(t)) = \frac{d}{dt} I_a \circ e^{t \vh} (p(t))  = (\vh I_a) (p(t)) = \sum_{i=1}^k a_i \{h, h_{i}\}(p(t)).
\ee
If $X = X_{jl}$, then the previous sum vanishes since $\{h_{jl}, h_i\} = 0$, see \eq{tab2}. And if $X=X_j$, then the sum~\eq{dI} is equal to 
$$
\sum_{i=1}^k a_i \{h_j, h_i\}(p(t)) = \sum_{i=1}^k a_i h_{ji}(p(t)) = (Ma)_j = 0 
$$
since $a \in \ker M_{p^0} = \ker M_{p(t)}$.

Thus $I_a(p(t)) \equiv \const$, so $I_a(p) = I_a(p^0)$. Property \eq{hijp} is proved in the special case~\eq{geTX}.

\medskip

(2) General case: any $g \in G$ can be represented as a product
$$
g = e^{T_N Y_N}  \cdots e^{T_1 Y_1}, \qquad T_i \in \R, \quad Y_i \in \{X_l, X_{lm}\}. 
$$ 
Then
$$
p = \Ad^*_{g^{-1}}(p^0) = e^{T_N \vg_N} \circ \cdots \circ e^{T_1 \vg_1}(p^0), 
$$
where $g_i(\lam) = \langle \lam, Y_i\rangle$, $\lam \in T^*G$. We use iteratively results of item (1) of this proof and obtain
\begin{align*}
&h_{ij} (p^0) = h_{ij}\circ e^{T_1 \vg_1}(p^0) = \cdots = h_{ij}\circ e^{T_N \vg_N} \circ \cdots  \circ e^{T_1 \vg_1}(p^0) = h_{ij}(p), \\
 &I_a (p^0) = I_a\circ e^{T_1 \vg_1}(p^0) = \cdots = I_a\circ e^{T_N \vg_N} \circ \cdots  \circ e^{T_1 \vg_1}(p^0) = I_a(p).
\end{align*}

We proved the inclusion $S_{p^0} \subset L_{p^0}$. Thus $S_{p^0}$ is open and closed in $L_{p^0}$. Since both $S_{p^0}$ and~$L_{p^0}$ are connected, we have $S_{p^0} = L_{p^0}$.
\end{proof}

\section{Casimir functions}\label{sec:Casimir}
Recall~\cite{kirillov}  that a {\em  Casimir function}  is a function  $h \in C^{\infty}(\g^*)$ such that
\be{}
\{h, h_i\} = \{h, h_{ij}\} = 0 \qquad \forall i, j. 
\ee
Symplectic leaves of maximal dimension are connected components of joint level surfaces of Casimir functions.

In view of \eq{tab2}, 
there are  $k(k-1)/2$  obvious Casimir functions 
\be{triv}
h_{ij}, 
\qquad 1 \leq i < j \leq k.
\ee 

\begin{theorem}
\begin{itemize}
\item[$(1)$]
If $k = 2 n$, $n \in \N$, then  there are just $k(k-1)/2$  independent Casimir functions  \eq{triv}.
\item[$(2)$]
If $k = 2 n+1$, $n \in \N$,  then, in addition to Casimir functions  \eq{triv}, there is one more independent of them:  
\begin{align}
&C(p) = I_{a(p)}(p) = \sum_{i=1}^k a_i(p) h_i(p), \label{C1} \\
\intertext{where}
&a_i = \sum_{\substack{j_l \neq i\\ j_1, \dots, j_{2n} \text{ different}}}
(-1)^{\sigma+i} h_{j_1j_2} \cdots h_{j_{2n-1}j_{2n}}, \qquad i = 1, \dots, k, \label{C2}\\
&\sigma = \text{ parity of the permutation } (j_1, \dots, j_{2n}). \label{C3}
 \end{align}
The Casimir function $C$ is a homogeneous polynomial of degree   $n+1  = (k+1)/2$  on $\g^*$.
\end{itemize}
\end{theorem}
\begin{proof}
(1) If $k = 2 n$, $n \in \N$, then, by Th. \ref{th:fol}, we have  $\dim S_{p^0} = k$ for generic $p^0 \in \g^*$, and there can be no more than $k(k-1)/2$ independent Casimir functions for dimension reason. Indeed, for generic $p^0 \in \g^*$, the codimension of $S_{p^0}$ in $\g^*$ is equal to the number of independent Casimir functions.

\medskip

(2) Let  $k = 2 n+1$, $n \in \N$. Then, by the same dimension reason, there must be one more Casimir function, independent of $h_{ij}$. Let us prove that it is the function $C$ given by \eq{C1}--\eq{C3}.

It is obvious that $\{C, h_{ij}\} = 0$ in view of \eq{tab2}.

Since $\{C,h_l\} = \sum_{i=1}^k a_i h_{il}$, it remains to prove that $S := \sum_{i=1}^k a_i h_{il} = 0$ for any $l = 1, \dots, k$.

Let $l = 1$ (the rest cases $l = 2, \dots, k$ are considered similarly). Then
\be{sum}
S = \sum_{i=1}^k a_i h_{i1} = \sum_{i=2}^k 
 \sum_{\substack{j_l \neq i\\ j_1, \dots, j_{2n} \text{ different}}}
(-1)^{\sigma+i} h_{j_1j_2} \cdots h_{j_{2n-1}j_{2n}}
h_{i1}.
\ee
Fix any $z  \in \{2, \dots, k\}$. 
The term of $S$ in \eq{sum} for $i = z$ contains a term for 
\be{ind1}
(j_1, j_2, \dots, j_{2m-1}, j_{2m}, \dots, j_{2n-1}, j_{2n})
\ee
 with $j_{2m-1} = 1$ or $j_{2m} = 1$ for some $m \in \{1, \dots n\}$. Let $j_{2m-1} = 1$ (the case $j_{2m} = 1$ is considered similarly). Then $j_{2m} \neq z$. 

Further, the term of $S$ in \eq{sum} for $i = j_{2m}$ contains a  term for 
\be{ind2}
(j'_1, j'_2, \dots, j'_{2m'-1}, j'_{2m'}, \dots, j'_{2n-1}, j'_{2n})
\ee 
with $(j'_{2m'-1}, j'_{2m}) = (z,1)$  (or $(j'_{2m'-1}, j'_{2m}) = (1,z)$, which is analogous)  for some $m' \in \{1, \dots n\}$.
It is easy to see that to each term in sum \eq{sum} with indices \eq{ind1} there corresponds the same term in \eq{sum} with indices \eq{ind2}, with the opposite sign. Thus $S= 0$. 
\end{proof}

\section{Pontryagin maximum principle}\label{sec:PMP}
We apply Pontryagin maximum principle  (PMP) in invariant form \cite{notes} to problem \eqref{p21}--\eqref{p23}. The control-dependent Hamiltonian for this problem is $\sum^k_{i=1} u_i h_i (\lambda)$, $\lambda \in T^*G$. The Hamiltonian system of PMP
reads
\begin{align}
&\dot{h}_i = -\sum^k_{j=1} u_j h_{ij}, \quad i = 1, \dots, k, \label{Ham1}& \\
&\dot{h}_{ij} = 0, \quad 1 \le i < j \le k, \label{Ham2}&\\
&\dot{q} = \sum^k_{i=1} u_i X_i, \label{Ham3}&
\end{align}
and the maximality condition of PMP is 
\begin{gather} \label{max}
\sum^k_{i=1} u_i (t) h_i (\lambda_t) = \underset{v \in U}{\mathrm{max}} \sum^k_{i=1} v_i h_i (\lambda_t) = H(h(\lambda_t)), 
\end{gather}
where
$$ H (h_1, \dots, h_k) := \underset{v \in U}{\mathrm{max}} \sum^k_{i=1} v_i h_i$$
is the support function of the set $U$~\cite{rock}. The function $H$ is convex, positive homogeneous, and continuous.
Along extremal trajectories we have $H \equiv \const \ge 0.$ 

The abnormal case 
$ H \equiv 0 \Leftrightarrow h_1 = \dots = h_k \equiv 0 $
can be omitted since the distribution  $\Delta = \spann(X_1, \dots, X_k)$ satisfies the condition $\Delta^2 = \Delta + [\Delta, \Delta] = TG$, thus by Goh condition \cite{notes} all locally optimal abnormal trajectories are simultaneously normal.

So we consider the normal case: $H \equiv \const > 0$. In view of homogeneity of the vertical part
\eqref{Ham1}, \eqref{Ham2} of the Hamiltonian system of PMP, we will assume that $H \equiv 1$ along extremal trajectories.

From now on we suppose additionally that the set $U$ is {\em strictly convex}. Then the maximized Hamiltonian $H$ is $C^1$-smooth on $\mathbb R^k \setminus \{0\}$, and maximum in \eqref{max} is attained at the control $u = \nabla H = (\partial H/\partial h_1, \dots, \partial H / \partial h_k)$ \cite{rock}. Denote $H_i = \partial H / \partial h_i$, $i = 1, \dots, k$. Then the vertical subsystem \eq{Ham1}, \eq{Ham2} of the Hamiltonian system of PMP reads as follows:
\begin{gather} \label{Hamax}
\dot{h}_i = -\sum^k_{j=1} h_{ij} H_j, \quad \dot{h}_{ij} = 0, \quad 1 \le i < j \le k, 
\end{gather}
where $H_j = H_j(h_1, \dots, h_k, h_{12}, \dots, h_{(k-1)k})$ are continuous functions.
In vector notation, this system reads as
\be{dotp}
\dot p = - M \nabla H(p), \qquad p \in \g^* \cap \{ h_{ij} \equiv \const\}.
\ee

\section{Integrals and solutions of Hamiltonian system}\label{sec:integrals}

\begin{theorem}\label{th:integral}
System \eq{Hamax}  has the following integrals:
\begin{itemize}
\item[$(1)$]
Casimir functions $h_{ij}$, $1 \leq i < j \leq k$,  and the Hamiltonian $H$,  for  $k = 2n$, $n \in \N$,
\item[$(2)$]
Casimir functions $h_{ij}$, $1 \leq i < j \leq k$,  and  $C$, and the Hamiltonian $H$,  for  $k = 2n + 1$, $n \in \N$.
\end{itemize}

For any $p^0 \in \g^*$, the symplectic leaf $S_{p^0}$ is an invariant set of system   \eq{Hamax}. In particular, any function   $I_a$, $a \in \ker M_{p^0}$, is constant on solutions of this system.  
\end{theorem}
\begin{proof}
It is well known that all Casimir functions and the Hamiltonian $H$ are integrals of the vertical subsystem  \eq{Hamax} of the Hamiltonian system corresponding to $H$ \cite{jurd}.
It is also a common knowledge that trajectories of this subsystem leave symplectic leaves invariant.  Finally, the constancy of $I_a$,  $a \in \ker M_{p^0}$, on these trajectories follows from Theorem \ref{th:fol} and the definition~\eq{Lp0} of the space $L_{p^0}$.
\end{proof}

\begin{remark}
If  $M = 0$, then solutions to system  \eq{Hamax} are constant, the corresponding normal extremal controls are constant and optimal.
\end{remark}

\begin{theorem}\label{th:p(t)}
Let $L$ be a step-$2$ free-nilpotent  Carnot algebra.
Suppose that $U$ is strictly convex and compact, and contains the origin in its interior.
Let $p^0 \in H^{-1}(1)$ and $\rank M_{p^0} = 2$.
Then the solution 
  $p(t)$ of system \eq{Hamax}  with the initial condition  $p(0) = p^0$  exists and is unique for any   $t \in \R$.
\begin{itemize}
\item[$(1)$]
If $\nabla H(p^0) \in \ker M$, then  $p(t) \equiv p^0$.
The corresponding normal extremal control 
 $u(t)$  is constant and optimal.
 \item[$(2)$]
If   $\nabla H(p^0) \not\in \ker M$, then  $p(t)$ is a strictly convex planar periodic regular $C^1$-smooth curve. The corresponding normal extremal control 
 $u(t)$  is periodic and continuous.
\end{itemize}
\end{theorem}

\begin{remark}
Notice that uniqueness of solutions to problem \eq{Hamax} is not immediate since the right-hand side of 
\eq{Hamax} is just continuous, but not $C^1$ or Lipschitzian.
Moreover, for nonstrictly convex $U$ (e.g. for $U$ polygon) we have non-uniqueness of solutions to Cauchy problem, when singular trajectories join bang-bang ones.
\end{remark}

We prove Th. \ref{th:p(t)}.

\begin{proof}


Consider any solution $p(t)$ to system \eq{Hamax}  with the initial condition  $p(0) = p^0$.
Such a solution exists and is defined for all $t \in \R$ since the level surface $H^{-1}(1)$ is compact.
 
Introduce an auxiliary convex optimization problem
\be{opt}
H(p) \to \min, \qquad p \in  S_{p^0}.
\ee
Since $H$ is continuous and nonnegative, and $\lim_{p\to \infty} H(p) = + \infty$, problem \eq{opt} has a solution:
$$
\exists \ H_{\min} = \min H|_{S_{p^0}} \geq 0.
$$
By Lagrange multipliers rule, a necessary and sufficient condition of minimum for problem \eq{opt} is as follows:
$$
H(p) = H_{\min} \ \Leftrightarrow \ \nabla H(p) \in \ker M \text{ or } H_{\min} = 0.
$$

\medskip

(1) Let $\nabla H(p^0) \in \ker M$. Then $H(p^0) = H_{\min}$, thus $H(p(t)) \equiv H(p^0) = H_{\min} = 1$. So $\nabla H(p(t)) \in \ker M$. 
Consequently, $\dot p(t) = - M \nabla H(p(t)) \equiv 0$, thus $p(t) \equiv p^0$.

\medskip

(2) Let $\nabla H(p^0) \not\in \ker M$.
Then $H(p^0) > H_{\min}$.

Consider a curve
$$
\Gamma  = \{ p \in \g^* \mid H(p) = H(p^0)\} \cap S_{p^0},
$$
it is obviously strictly convex, compact and planar. Moreover, for any $p \in \Gamma$ we have
$$
H(p) = H(p^0) > H_{\min} \ \Rightarrow \ \nabla H(p) \not\in \ker M.
$$
Thus $\Gamma$ is a regular $C^1$-smooth curve diffeomorphic to $S^1$.

Choose coordinates in $\g^* \cap \{h_{ij} = h_{ij}(p^0) \} \cong \mathbb R_{h_1, \dots, h_k}^k$ such that $S_{p^0} = \{ (h_3, \dots, h_{k}) = \const \}$ (we keep the old notation $(h_1, \dots, h_k)$ for the new coordinates). Parametrize the curve $\Gamma$ as follows: 
$$
h_1 = f_1(\f), \quad h_2 = f_2(\f), \quad (h_3, \dots, h_{k}) \equiv \const, \qquad \f \in S^1,
$$
where $f_1, f_2 \in C(S^1)$. In this parametrization ODE~\eqref{Hamax} reads as
\be{dotf}
\dot \f = f(\f), \qquad \f \in S^1,
\ee
where $f \in C(S^1)$ and $f(\f) \neq 0$ for all $\f \in S^1$. ODE~\eqref{dotf} can uniquely be solved for any initial data by separation of variables, thus it has a unique solution $\f(t) \in C^1(S^1)$ for any Cauchy problem $\f(0) = \f^0$. Thus ODE~\eqref{Hamax} has also a unique solution $h(t) \in C^1$ for the Cauchy problem $h(0) = h^0$.

Further, $p(t) \in \Gamma$ and $\dot p(t) = -M \nabla H(p(t)) \neq 0$ for all $t$, thus there exists $T > 0$ such that $p(T) = h^0$. By uniqueness of $p(t)$, it is $T$-periodic.
\end{proof}

\begin{corollary}\label{cor:phase}
Let $L$ be a step-$2$ free-nilpotent  Carnot algebra.
Suppose that $U$ is strictly convex and compact, and contains the origin in its interior.
Let $p_0 \in \g^*$, and let  $\dim S_{p^0}= 2$. 
Then the system \eq{Hamax} has the following phase portrait on the leaf $S_{p^0}$, see Figs.~$\ref{fig:ellipses}$--$\ref{fig:ellipses3}$.
There is a closed convex nonempty subset $D_0 \subset S_{p^0}$ such that:
\begin{itemize}
\item[$(1)$]
any solution $p(t)$ to \eq{Hamax} with initial condition  $p(0) \in D_0$ is constant,
\item[$(2)$]
any solution $p(t)$ to \eq{Hamax} with initial condition  $p(0) \in S_{p^0} \setminus D_0$ is periodic and encircles the set $D^0$.
\end{itemize}
\end{corollary}
\begin{proof}
$D_0 = \{p \in  S_{p^0} \mid \nabla H(p) \in \ker M_p\}$.
\end{proof}

\begin{figure}[htbp]
\includegraphics[width=0.3\textwidth]{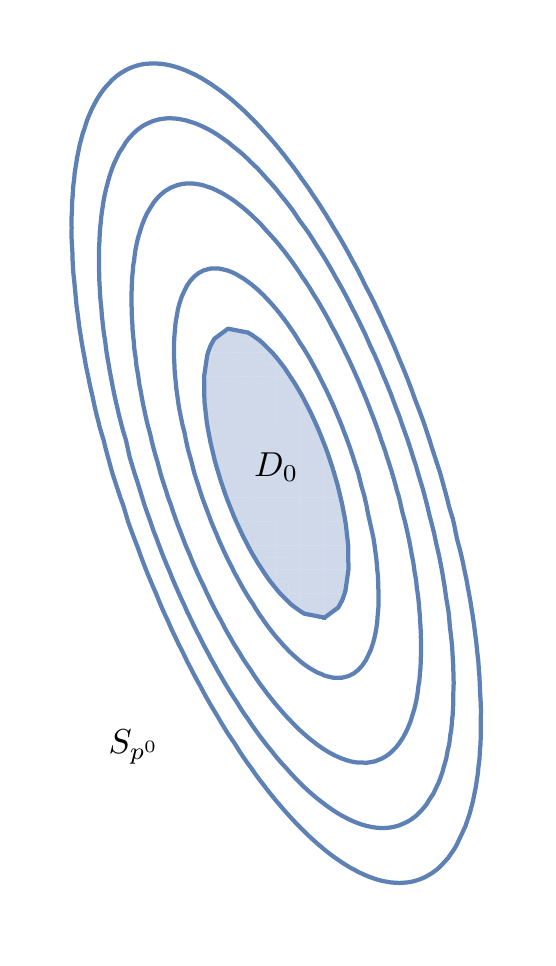}
\hfill
\includegraphics[width=0.3\textwidth]{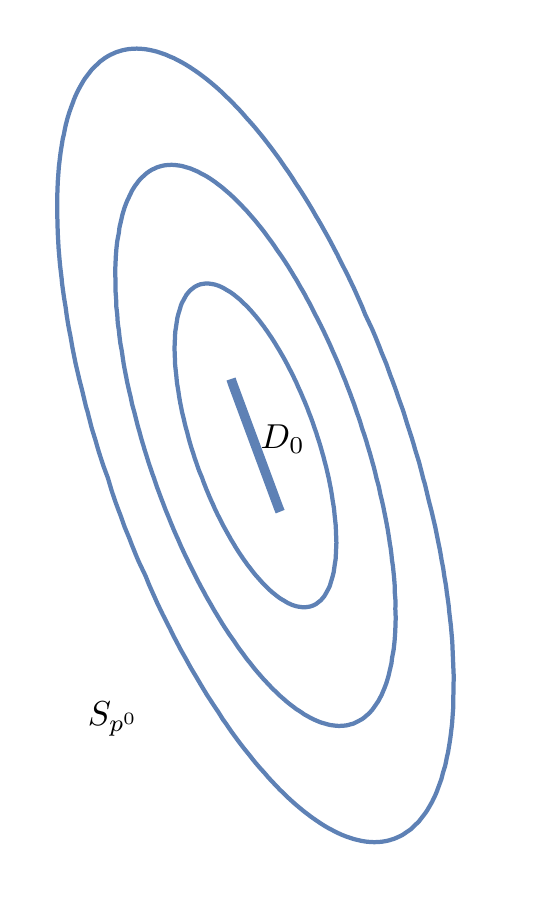}
\hfill
\includegraphics[width=0.3\textwidth]{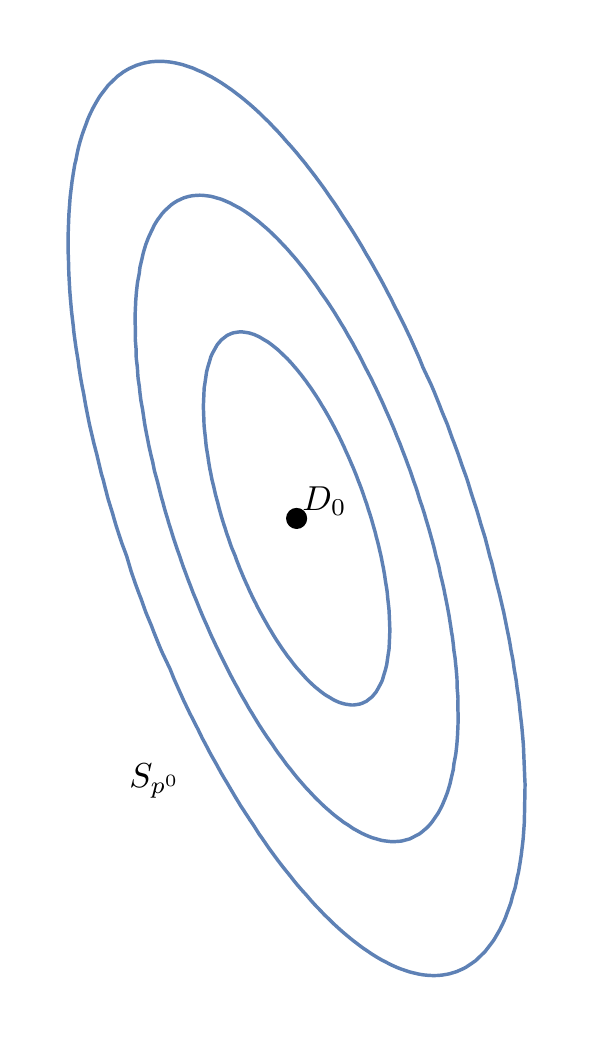}
\\
\parbox[t]{0.3\textwidth}{\caption{Phase portrait of system \eq{Hamax}  on the leaf $S_{p^0}$ with $\dim D_0 = 2$}\label{fig:ellipses}}
\hfill
\parbox[t]{0.3\textwidth}{\caption{Phase portrait of system \eq{Hamax}  on the leaf $S_{p^0}$  with $\dim D_0 = 1$}\label{fig:ellipses2}}
\hfill
\parbox[t]{0.3\textwidth}{\caption{Phase portrait of system \eq{Hamax}  on the leaf $S_{p^0}$  with $\dim D_0 = 0$}\label{fig:ellipses3}}
\end{figure}

\begin{remark}
Unlike the set $U$, its polar set $U^{\circ} = \{ p \in \g^* \mid H(p) \leq 1\}$ is not {\em strictly} convex, thus its boundary may have faces of various dimensions. If such a face  intersects a leaf $S_{p^0}$ by a set of dimension $2$, $1$, or $0$, then the set $D_0$ from Cor. $\ref{cor:phase}$ coincides with  the intersection of the face  with $S_{p^0}$ of the same dimension $2$, $1$, or $0$, see resp. Figures $\ref{fig:ellipses}$, $\ref{fig:ellipses2}$, or $\ref{fig:ellipses3}$.
\end{remark}

\begin{corollary}
Let $\g$ be a step-$2$ Carnot algebra.
  Suppose that $U$ is strictly convex and compact, and contains the origin in its interior.
Let $p^0 \in H^{-1}(1)$ and $\rank M_{p^0} = 2$. Then the corresponding normal extremal control for problem \eq{p21}--\eq{p23} is constant or periodic  and continuous.
\end{corollary}

\section{Final remarks}\label{sec:final}
The results of this paper clarify  the first floors of an infinite {\em hierarchy} of qualitative behaviour of extremal controls (and solutions of the vertical subsystem of the Hamiltonian system of PMP) for left-invariant time-optimal problems \eq{p21}--\eq{p23} on  two-step Carnot groups, labelled by the dimension $d$ of co-adjoint orbits containing these solutions:
\begin{itemize}
\item
for $d=0$ the controls are constant (thus optimal),
\item
for $d=2$ the controls are constant or periodic,
\item
for $d \geq 4$, a chaotic behaviour is generic, as suggested by the next theorem.
\end{itemize}

\begin{theorem}\label{th:SR}
Let  $U = \left\{ \sum_{i=1}^k u_i^2 \leq 1 \right\}$. 
If dimension of a co-adjoint orbit is greater than $2$, then generic solutions to system 
 \eq{Hamax} contained in such orbit are non-periodic and have a full-dimensional closure.
\end{theorem}
\begin{proof}
In the sub-Riemannian case system \eq{dotp} reads $\dot p = - 2 M p$. The skew-symmetric matrix~$M$ has eigenvalues    
$\pm i \a_1$, \dots, $\pm i \a_n$ and 0 for the case of odd $k$. Thus $e^{-2tM}$ is periodic iff
$m_1 \a_1 =  \cdots =  m_n \a_n$,  $m_j \in \N$.  If neither of these equalities holds, a trajectory $p(t) = e^{-2tM}p^0$ has full-dimensional closure, thus is not periodic.
\end{proof}

In fact, the above hierarchy is quite natural in view of the general theory of Hamiltonian systems of ODEs. Recall that restriction of a Hamiltonian system to a co-adjoint orbit is Hamiltonian, and that Hamiltonian systems with two degrees of freedom are Liouville integrable. Although, this theory is not applicable directly since the right-hand side of system~\eq{Hamax} is just continuous, but not smooth. As we already mentioned, even uniqueness of solutions to~\eq{Hamax} is not straightforward. 

A natural goal is to describe the qualitative behaviour of solutions to system~\eq{Hamax} for the next floors of the hierarchy with $d \geq 4$.

For Carnot groups of step $s > 2$, the picture becomes more complicated.

In the free-nilpotent case with  $s= 3$, $k = 2$ (the Cartan group), for the sub-Riemannian problem $U = \{u_1^2+u_2^2 \leq 1\}$, optimal controls are given by Jacobi's elliptic functions~\cite{dido_exp}, and they are of the following classes:
\begin{itemize}
\item
constant for $d = 0$,
\item
constant, periodic or  asymptotically constant (with constant limits as $t \to \pm \infty$) for $d = 2$.
\end{itemize}

It would be interesting to characterize similarly optimal controls in the cases $s=3$, $k \geq 3$ and $s\geq 4$, at least for $d=2$. In these cases, if $U = \{\sum_{i=1}^k u_i^2 \leq 1\}$, the normal Hamiltonian system of Pontryagin maximum principle is not Liouville integrable~\cite{borisov, 2358}.

\end{document}